\documentclass[12pt]{article}

\usepackage{amsthm,amsmath,amssymb}
\usepackage{graphicx}
\usepackage[colorlinks=true,citecolor=black,linkcolor=black,urlcolor=blue]{hyperref}

\usepackage{mathtools}
\usepackage{tikz}
\usepackage{tkz-berge}
\usepackage{hyperref}
\usepackage{enumerate}
\usepackage{microtype}

\usepackage{longtable}
\theoremstyle{plain}
\newtheorem{theorem}{Theorem}
\newtheorem{lemma}[theorem]{Lemma}
\newtheorem{corollary}[theorem]{Corollary}

\theoremstyle{definition}
\newtheorem{definition}[theorem]{Definition}

\theoremstyle{remark}

\title{\bf New graceful diameter-6 trees\\by transfers}

\author{Matt Superdock\thanks{This paper is based on the author's undergraduate senior thesis; he graduated in 6/13.}\\
\small Department of Mathematics\\[-0.8ex]
\small Princeton University\\[-0.8ex] 
\small Princeton, NJ, USA\\
\small\tt msuperdock@gmail.com\\
}

\date{\small Jun 27, 2015}

\begin{document}

\maketitle

\begin{abstract}

Given a graph $G$, a \emph{labeling} of $G$ is an injective function $f:V(G)\rightarrow\mathbb{Z}_{\ge 0}$. Under the labeling $f$, the \emph{label} of a vertex $v$ is $f(v)$, and the \emph{induced label} of an edge $uv$ is $|f(u) - f(v)|$.  The labeling $f$ is \emph{graceful} if the labels of the vertices are $\{0, 1, \ldots , |V(G)| - 1\}$, and the induced labels of the edges are distinct. The graph $G$ is \emph{graceful} if it has a graceful labeling. The Graceful Tree Conjecture, introduced by Kotzig in the late 1960's, states that all trees are graceful. It is an open problem whether every diameter-6 tree has a graceful labeling. In this paper, we prove that if $T$ is a tree with central vertex and root $v$, such that each vertex not in the last two levels has an odd number of children, and $T$ satisfies one of the following conditions (a)-(e), then $T$ has a graceful labeling $f$ with $f(v) = 0$: (a) $T$ is a diameter-6 complete tree; (b) $T$ is a diameter-6 tree such that no two leaves of distance 2 from $v$ are siblings, and each leaf of distance 2 from $v$ has a sibling with an even number of children; (c) $T$ is a diameter-$2r$ complete tree, such that the number of vertices of distance $r - 1$ from $v$, with an even number of children, is not $3\pmod{4}$; (d) $T$ is a diameter-$2r$ tree, such that the number of vertices of distance $r - 1$ from $v$, with an even number of children, is not $3\pmod{4}$, no two leaves of distance $r - 1$ from $v$ are siblings, and each leaf of distance $r - 1$ from $v$ has a sibling with an even number of children; (e) $T$ is a diameter-6 tree, such that each internal vertex has an odd number of children. In particular, all depth-3 trees of which each internal vertex has an odd number of children are graceful.


\end{abstract}

\section{Introduction}
\label{}

Given a graph $G$, a \emph{labeling} of $G$ is an injective function $f:V(G)\rightarrow\mathbb{Z}_{\ge 0}$.  Under the labeling $f$, the \emph{label} of a vertex $v$ is $f(v)$, and the \emph{induced label} of an edge $uv$ is $|f(u) - f(v)|$.  The labeling $f$ is \emph{graceful} if:
\begin{itemize}
\item The labels of the vertices are $\{0, 1, \ldots , |V(G)| - 1\}$, and
\item The induced labels of the edges are distinct.
\end{itemize}
Note that if $G$ is a tree, the induced labels must be $\{1, \ldots , |V(G)| - 1\}$.  The graph $G$ is \emph{graceful} if it has a graceful labeling.  The Graceful Tree Conjecture posits that all trees are graceful (Kotzig, see Bermond~\cite{survey0}). See Gallian's survey~\cite{survey4} for the current state of progress toward the conjecture.

In 2001, Hrn\v{c}iar \& Haviar~\cite{diameter5} introduced the idea of a \emph{transfer}, which they used to prove that all diameter-5 trees are graceful. This idea is best understood by example; see Fig.~\ref{fig1}. Each step of Fig.~\ref{fig1} is a \emph{transfer} of leaves, though Hrn\v{c}iar \& Haviar~\cite{diameter5} also introduce transfers of branches.  They established the following now-standard conventions:
\begin{itemize}
\item We identify vertices with labels; ``$k$" can refer to the vertex labeled $k$.
\item We use the notation $i\rightarrow j$ to represent a transfer from (vertex) $i$ to $j$.
\end{itemize}

\begin{figure}[!ht]
\begin{center}
\beginpgfgraphicnamed{transfer}
\begin{tikzpicture}[rounded corners=5pt, >=latex]

\node at (0.4, 17.1) {1.};
\node at (0.4, 11.6) {2.};
\node at (6.4, 17.1) {3.};
\node at (6.4, 11.6) {4.};
\node at (2.65, 6.1) {5.};

\draw (0,12.5) rectangle (5.5, 17.5);		%for graph 1
\draw (0, 7) rectangle (5.5, 12);			%for graph 2
\draw (6,12.5) rectangle (13.5, 17.5);		%for graph 3
\draw (6, 7) rectangle (13.5, 12);			%for graph 4
\draw (2.25, -0.5) rectangle (11.25, 6.5);	%for graph 5

%graph1
\begin{scope}[shift={(0.5,14)}]
\Vertex[x=0,y=0]{1}
\Vertex[x=1,y=0]{2}
\draw[loosely dotted, very thick] (1.5,0) -- (1.9,0);
\Vertex[x=2.4,y=0]{10}
\Vertex[x=3.5,y=0]{11}
\Vertex[x=4.5,y=0]{12}
\Vertex[x=2.25,y=2]{0}
\Edge[label=$1$](0)(1)
\Edge[label=$12$](0)(12)
\Edge[label=$2$](0)(2)
\Edge[label=$10$](0)(10)
\Edge[label=$11$](0)(11)
\end{scope}

%graph2
\begin{scope}[shift={(0.5,7.5)}]
\Vertex[x=0,y=0]{2}
\Vertex[x=1,y=0]{3}
\draw[loosely dotted, very thick] (1.5,0) -- (2,0);
\Vertex[x=2.5,y=0]{10}
\Vertex[x=1.25,y=2]{12}
\Vertex[x=2.25,y=2]{1}
\Vertex[x=3.25,y=2]{11}
\Vertex[x=2.25,y=4]{0}
\Edge[label=$10$](12)(2)
\Edge[label=$2$](12)(10)
\Edge[label=$12$](0)(12)
\Edge[label=$1$](0)(1)
\Edge[label=$11$](0)(11)
\Edge[label=$9$](12)(3)
\end{scope}

%graph3
\begin{scope}[shift={(7,13)}]
\Vertex[x=-0.5,y=0]{2}
\Vertex[x=0.5,y=0]{3}
\Vertex[x=1.5,y=0]{4}
\draw[loosely dotted, very thick] (2,0) -- (2.5,0);
\Vertex[x=3,y=0]{8}
\Vertex[x=4,y=0]{9}
\Vertex[x=5,y=0]{10}
\Vertex[x=0.25,y=2]{12}
\Vertex[x=2.75,y=2]{1}
\Vertex[x=5.25,y=2]{11}
\Vertex[x=2.75,y=4]{0}
\Edge[label=$10$](12)(2)
\Edge[label=$2$](1)(3)
\Edge[label=$9$](1)(10)
\Edge[label=$12$](0)(12)
\Edge[label=$1$](0)(1)
\Edge[label=$11$](0)(11)
\Edge[label=$3$](1)(4)
\Edge[label=$7$](1)(8)
\Edge[label=$8$](1)(9)
\end{scope}

%graph4
\begin{scope}[shift={(6.5,7.5)}]
\Vertex[x=0,y=0]{2}
\Vertex[x=1,y=0]{10}
\Vertex[x=2,y=0]{3}
\Vertex[x=3,y=0]{9}
\Vertex[x=4,y=0]{4}
\Vertex[x=5,y=0]{5}
\draw[loosely dotted, very thick] (5.5,0) -- (6,0);
\Vertex[x=6.5,y=0]{8}
\Vertex[x=0.75,y=2]{12}
\Vertex[x=3.25,y=2]{1}
\Vertex[x=5.75,y=2]{11}
\Vertex[x=3.25,y=4]{0}
\Edge[label=$10$](12)(2)
\Edge[label=$2$](1)(3)
\Edge[label=$9$](1)(10)
\Edge[label=$8$](1)(9)
\Edge[label=$7$](11)(4)
\Edge[label=$3$](11)(8)
\Edge[label=$12$](0)(12)
\Edge[label=$1$](0)(1)
\Edge[label=$11$](0)(11)
\Edge[label=$6$](11)(5)
\end{scope}

%graph5
\begin{scope}[shift={(2.75,0)}]
\Vertex[x=0,y=0]{5}
\Vertex[x=1,y=0]{6}
\Vertex[x=2,y=0]{7}
\Vertex[x=3,y=0]{8}
\Vertex[x=1.5,y=2]{2}
\Vertex[x=3,y=2]{10}
\Vertex[x=4,y=2]{3}
\Vertex[x=5,y=2]{9}
\Vertex[x=6.5,y=2]{4}
\Vertex[x=1.5,y=4]{12}
\Vertex[x=4,y=4]{1}
\Vertex[x=6.5,y=4]{11}
\Vertex[x=4,y=6]{0}
\Edge[label=$12$](0)(12)
\Edge[label=$1$](0)(1)
\Edge[label=$11$](0)(11)
\Edge[label=$10$](12)(2)
\Edge[label=$9$](1)(10)
\Edge[label=$2$](1)(3)
\Edge[label=$8$](1)(9)
\Edge[label=$7$](11)(4)
\Edge[label=$3$](2)(5)
\Edge[label=$4$](2)(6)
\Edge[label=$5$](2)(7)
\Edge[label=$6$](2)(8)
\end{scope}
\end{tikzpicture}

\endpgfgraphicnamed
\end{center}
\caption{Performing a sequence of transfers on a gracefully labeled star.}
\label{fig1}
\end{figure}
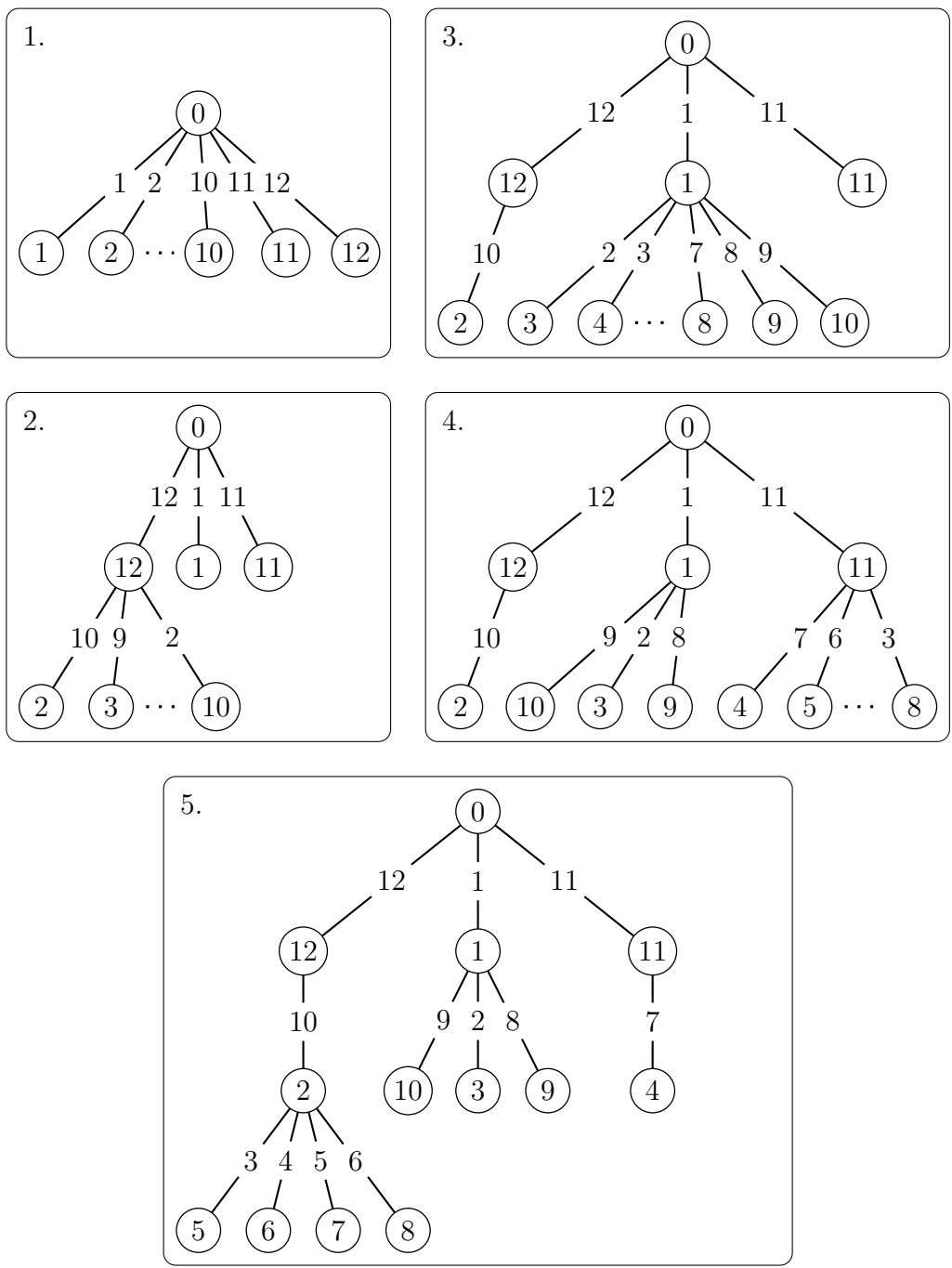

For example, we describe the transfers in Fig.~\ref{fig1} as follows:
$$0\rightarrow 12\rightarrow 1\rightarrow 11\rightarrow 2$$
Our example (Fig.~\ref{fig1}) reflects the standard strategy of Hrn\v{c}iar \& Haviar~\cite{diameter5}, of beginning with a gracefully labeled star and using the following transfers:
$$0\rightarrow n\rightarrow 1\rightarrow n - 1\rightarrow\cdots$$
As in Fig.~\ref{fig1}, this sequence of transfers builds trees in a breadth-first way (see Lemma~\ref{order}). This strategy has since been applied to many other classes of trees (see \cite{banana, fixed-banana, jesinthageneration, mishra2005graceful, mishra2006graceful, mishra2008some, mishra2010some, mishra2011some, sethuraman1, sethuraman3, sethuraman2}). In this paper we expand the strategy considerably, yielding new best results on diameter-6 trees.

\subsection{Type-1 and Type-2 Transfers}

Hrn\v{c}iar \& Haviar~\cite{diameter5} introduced two basic types of transfers of leaves.  Their definitions and important properties (assuming the standard sequence of transfers $0\rightarrow n\rightarrow 1\rightarrow\cdots$) are as follows:
\begin{itemize}
\item A \emph{type-1 transfer} is a transfer $u\rightarrow v$ of leaves $k$, $k + 1, \ldots , k + m$, possible if $f(u) + f(v) = k + (k + m)$.  A type-1 transfer leaves behind an odd number of leaves and can be followed by a type-1 or type-2 transfer.
\item A \emph{type-2 transfer} is a transfer $u\rightarrow v$ of leaves $k$, $k + 1, \ldots , k + m$, $l$, $l + 1, \ldots , l + m$, possible if $f(u) + f(v) = k + (l + m)$.  A type-2 transfer leaves behind an odd or even number of leaves, but must leave behind an even number if it follows a type-2 transfer, and can only be followed by a type-2 transfer.
\end{itemize}
In Fig.~\ref{fig1}, all four transfers are type-1, leaving behind odd numbers of leaves.  From the final tree, we could, e.g., perform a type-2 transfer $2\rightarrow 10$, transferring the leaves $5, 7$, and leaving the vertices $2, 10$ each with two children.

In the same way, the standard approach is to perform type-1 transfers to obtain some vertices with odd numbers of children, and then switch to type-2 transfers to obtain some vertices with even numbers of children. However, in the course of this paper we show that the results of this approach can be replicated using only type-1 transfers, implying that type-2 transfers are not strictly necessary.

\subsection{Rearranging Subtrees}

Using transfers often requires rearranging the subtrees of a tree, as in the cases of \emph{banana trees} and \emph{generalized banana trees}.

A \emph{banana tree} is a tree obtained from a collection of stars $K_{1, m_{i}}$ with $m_{i}\ge 1$ by joining a leaf of each star by an edge to a new vertex $v$, called the \emph{apex}.  A \emph{generalized banana tree} is obtained from a banana tree by replacing the edges incident with $v$ with paths of fixed length $h\ge 0$.

\begin{theorem}
{\normalfont (Hrn\v{c}iar \& Monoszova~\cite{banana})} Let $T$ be a generalized banana tree, with apex $v$ of odd degree.  Then $T$ has a graceful labeling $f$ with $f(v) = 0$.
\end{theorem}

\begin{proof}
Rearrange subtrees so that the stars occur in the following order:
\begin{enumerate}
\item Stars $K_{1, m_{i}}$ with $m_{i}$ even.
\item Stars $K_{1, m_{i}}$ with $m_{i}$ odd and $m_{i} > 1$.
\item Stars $K_{1, m_{i}}$ with $m_{i} = 1$.
\end{enumerate}
Consider a gracefully labeled star with central vertex labeled 0, and apply type-1 transfers $0\rightarrow n\rightarrow 1\rightarrow\cdots$, switching to type-2 at the last level.
\end{proof}

By attaching a caterpillar (see Lemma~2 of Hrn\v{c}iar \& Haviar~\cite{diameter5}), we get:

\begin{theorem}
{\normalfont (Hrn\v{c}iar \& Monoszova~\cite{banana})} All generalized banana trees are graceful.
\end{theorem}

Jesintha \& Sethuraman~\cite{fixed-banana, jesinthageneration, sethuraman1, sethuraman3, sethuraman2} have generalized this result further, reapplying the strategy of rearranging subtrees of a tree at the apex. In this paper, we attain new results by rearranging smaller subtrees of a tree.

\subsection{Alternative Sequences of Transfers}

In order to overcome the limitations of type-1 and type-2 transfers, Hrn\v{c}iar \& Haviar~\cite{diameter5} introduced the \emph{backwards double 8 transfer}, an alternative sequence of type-1 transfers
$$0\rightarrow n\rightarrow 1\rightarrow n - 1\rightarrow 0\rightarrow n\rightarrow 1\rightarrow n - 1\rightarrow 2$$
This sequence of transfers leaves an even number of leaves at each of $0$, $n$, $1$, $n - 1$, and can be followed by returning to the standard sequence of transfers $2\rightarrow n - 2\rightarrow 3\rightarrow\cdots$ with type-1 or type-2 transfers.  Using this idea, Hrn\v{c}iar \& Haviar~\cite{diameter5} prove their main result:

\begin{theorem}
{\normalfont (Hrn\v{c}iar \& Haviar~\cite{diameter5})} All diameter-5 trees are graceful.
\end{theorem}

In this paper, we introduce several new and more general sequences of transfers.

\subsection{New Results}

By extending the transfer technique as noted above, we prove the following. We refer to a rooted tree with leaves only in the last level as a \emph{complete} tree.

\begin{theorem} \label{main_theorem}
Let $T$ be a tree with central vertex and root $v$, such that each vertex not in the last (farthest from the root) two levels has an odd number of children.  Also, suppose $T$ satisfies one of the following sets of conditions:
\begin{enumerate}[(a)]
\item $T$ is a diameter-6 complete tree.
\item $T$ is a diameter-6 tree, such that
\begin{itemize}
\item No two leaves of distance 2 from $v$ are siblings.
\item Each leaf of distance 2 from $v$ has a sibling with an even number of children.
\end{itemize}
\item $T$ is a diameter-$2r$ complete tree, such that the number of vertices of distance $r - 1$ from $v$, with an even number of children, is not $3\pmod{4}$.
\item $T$ is a diameter-$2r$ tree, such that the number of vertices of distance $r - 1$ from $v$, with an even number of children, is not $3\pmod{4}$, and
\begin{itemize}
\item No two leaves of distance $r - 1$ from $v$ are siblings.
\item Each leaf of distance $r - 1$ from $v$ has a sibling with an even number of children.
\end{itemize}
\item $T$ is a diameter-6 tree, such that each internal vertex has an odd number of children.
\end{enumerate}
Then $T$ has a graceful labeling $f$ with $f(v) = 0$.
\end{theorem}

We can attach an arbitrary caterpillar (see Lemma~2 of Hrn\v{c}iar \& Haviar~\cite{diameter5}) to the root of any of these trees and obtain another graceful tree. In particular, this gives the following nice class of trees:

\begin{corollary}
Let $T$ be a diameter-6 tree with central vertex and root $v$, such that each internal vertex has an odd number of children.  Then $T$ has a graceful labeling $f$ with $f(v) = 0$.
\end{corollary}

\begin{proof}
Starting with a tree from Thm.~\ref{main_theorem}(e), attach leaves to the root.
\end{proof}

%%%%%%%%%%%%%%%%%%%%%%%%%%%%%%%%%%%%%%%%%%%%%%%%%%%%%%%

\section{Preliminary Results}
\label{preliminary}

We now work toward defining \emph{attainable} and \emph{nicely attainable} sequences.  Essentially,
\begin{itemize}
\item
An \emph{attainable} sequence is a list of numbers of leaves that can be left behind by a \emph{well-behaved} sequence of transfers.
\item
A \emph{nicely attainable} sequence is a list of numbers of leaves that can be left behind by a \emph{well-behaved} sequence of transfers, such that the sequence of transfers can continue onward.
\end{itemize}

\begin{definition} \label{transfer_context}
Let $T$ be a tree, let $f$ be a graceful labeling of $T$, let $v_{1}, \ldots , v_{m}$ be a sequence of vertices of $T$, and let $a, b, c, d$ be integers.  We say that $(T, f, v_{1}, \ldots , v_{m}, a, b, c, d)$ is a \emph{transfer context} if the following conditions hold:
\begin{itemize}
\item The labels of $v_{1}, \ldots , v_{m}$, in order, are
$$a,\; b - 1,\; a + 1,\; b - 2,\ldots\qquad\text{or}\qquad a,\; b + 1,\; a - 1,\; b + 2,\ldots$$
\item $v_{1}$ is adjacent to leaves with labels $c, \; c + 1, \ldots ,\; d$.\\
($v_{1}$ may also be adjacent to other leaves.)
\item $a + b = c + d$.\\
(This ensures that the labels of the leaves at $v_{1}$ are the same as if the leaves were just transferred from a previous vertex $v_{0}$ with label $b$.)

\end{itemize}
\end{definition}

The two possible sequences of labels share the essential characteristic of the standard sequence $0\rightarrow n\rightarrow 1\rightarrow\cdots$, that the sum of adjacent labels alternates between constants $k$, $k + 1$ at each step.

\begin{definition}
Given a transfer context $(T, f, v_{1}, \ldots , v_{m}, a, b, c, d)$, a \emph{well-behaved} sequence of transfers is a finite sequence of transfers
$$v_{i_{1}}\rightarrow v_{i_{2}}\rightarrow\cdots\rightarrow v_{i_{k}}$$
with $i_{1} = 1$, such that
\begin{itemize}
\item All transfers are type-1 transfers.
\item The set of leaves transferred in the first step is a subset of the set of leaves $\{c,\ldots , d\}$.  The set of leaves transferred in each subsequent step is a subset of the leaves transferred in the previous step.
\item Each transfer $v_{i_{j}}\rightarrow v_{i_{j + 1}}$ has indices $i_{j}, i_{j + 1}$ of different parity.
\end{itemize}
The \emph{result} of a well-behaved sequence of transfers is the sequence $n_{1}, \ldots , n_{m}$ of non-negative integers, such that after performing the transfers, each $v_{k}$ is adjacent to $n_{k}$ of the leaves $c,\ldots , d$ originally adjacent to $v_{1}$.
\end{definition}

\begin{definition}
A sequence $n_{1}, \ldots , n_{m}$ of non-negative integers is \emph{attainable} if, for any transfer context $(T, f, v_{1}, \ldots , v_{m}, a, b, c, d)$ with
$$n_{1} + \cdots + n_{m} = |\{c, \ldots , d\}|,$$
there exists a well-behaved sequence of transfers with result $n_{1}, \ldots , n_{m}$.
\end{definition}

\begin{definition}
A sequence $n_{1}, \ldots , n_{m}$ of non-negative integers is \emph{nicely attainable} if, for any positive integer $n_{m + 1}$, and for any transfer context $(T, f, v_{1}, \allowbreak \ldots , v_{m + 1}, a, b, c, d)$ with
$$n_{1} + \cdots + n_{m + 1} = |\{c, \ldots , d\}|,$$
there exists a well-behaved sequence of transfers with result $n_{1}, \ldots , n_{m + 1}$, such that the last transfer is $v_{m}\rightarrow v_{m + 1}$, and $v_{m + 1}$ occurs in no other transfer.
\end{definition}

\begin{lemma}\label{combine_attainable}
The following results hold:
\begin{enumerate}[(a)]
\item If the sequence $n_{1}, \ldots , n_{m}$ is nicely attainable, and the sequence $n_{1}', \ldots , n_{m'}'$ is attainable, then the sequence $n_{1}, \ldots , n_{m}, n_{1}', \ldots , n_{m'}'$ is attainable.
\item If the sequence $n_{1}, \ldots , n_{m}$ is nicely attainable, and the sequence $n_{1}', \ldots , n_{m'}'$ is nicely attainable, then the sequence $n_{1}, \ldots , n_{m}, n_{1}', \ldots , n_{m'}'$ is nicely attainable.
\end{enumerate}
\end{lemma}

\begin{proof}
Concatenate the two corresponding sequences of transfers.
\end{proof}

\subsection{New Sequences of Transfers}

For the standard sequence of transfers $0\rightarrow n\rightarrow 1\rightarrow\cdots$, we can leave behind any odd number of leaves at each step (see Fig.~\ref{fig1}), essentially because the sum of adjacent labels changes (increases or decreases) by one at each step.  In general, if the sum changes by $k$, then we can leave behind $l$ leaves, where $l$ is any integer with $l\ge k$ and $l\equiv k\pmod{2}$.  The following lemma makes this idea precise:

\begin{lemma} \label{transfer_parity}
Let $(T, f, v_{1}, \ldots , v_{m}, a, b, c, d)$ be a transfer context.  Consider performing a well-behaved sequence of transfers, beginning with
$$v_{i_{1}}\rightarrow\cdots\rightarrow v_{i_{j}},\qquad i_{1} = 1, \qquad j > 1.$$
Let $v_{i_{j + 1}}$ be a vertex with $i_{j + 1}\not\equiv i_{j}\pmod{2}$, and let $l$ be an integer such that the previous transfer $v_{i_{j - 1}}\rightarrow v_{i_{j}}$ transfers at least $l$ leaves, and such that $l\ge |i_{j + 1} - i_{j - 1}|/2$ and $l\equiv |i_{j + 1} - i_{j - 1}|/2\pmod{2}$.  Then there exists a transfer $v_{i_{j}}\rightarrow v_{i_{j + 1}}$ of the first type, leaving behind $l$ leaves at $v_{i_{j}}$.
\end{lemma}

\begin{proof}
Since $(T, f, v_{1}, \ldots , v_{m}, a, b, c, d)$ is a transfer context, we have $|f(v_{i_{j + 1}}) - f(v_{i_{j - 1}})| = |i_{j + 1} - i_{j - 1}|/2$. Suppose the previous transfer $v_{i_{j - 1}}\rightarrow v_{i_{j}}$ transfers leaves with consecutive labels $c', \ldots , d'$, so that $f(v_{i_{j - 1}}) + f(v_{i_{j}}) = c' + d'$.  Then
$$f(v_{i_{j}}) + f(v_{i_{j + 1}}) = c' + d' \pm |i_{j + 1} - i_{j - 1}|/2$$
Therefore, a transfer $v_{i_{j}}\rightarrow v_{i_{j + 1}}$ of the first type transfers at most the leaves with labels (depending on the sign in the expression above)
$$\{c' + |i_{j + 1} - i_{j - 1}|/2, \ldots , d'\} \qquad\text{or}\qquad\{c', \ldots , d' - |i_{j + 1} - i_{j - 1}|/2\}.$$
Therefore, we can leave behind $|i_{j + 1} - i_{j - 1}|/2$ leaves at $v_{i_{j}}$.  By removing $k$ elements from each end of the relevant set above, we can more generally leave behind $l = |i_{j + 1} - i_{j - 1}|/2 + 2k$ leaves for any $k\ge 0$, as long as there are at least $l$ leaves at $v_{i_{j}}$ to leave behind.
\end{proof}

Taking $i_{0} = 0$ (see end of Df.~\ref{transfer_context}) gives the analogous result for $j = 1$:

\begin{lemma}
Let $(T, f, v_{1}, \ldots , v_{m}, a, b, c, d)$ be a transfer context.  Let $v_{i_{2}}$ be a vertex with $i_{2}$ even, and let $l$ be an integer such that $\{c, \ldots , d\}$ is a set of at least $l$ leaves, and such that $l\ge |i_{2}|/2$ and $l\equiv |i_{2}|/2\pmod{2}$.  Then there exists a transfer $v_{1}\rightarrow v_{i_{2}}$ of the first type, leaving behind $l$ leaves at $v_{1}$.
\end{lemma}

\begin{proof}
Analogous to proof of Lemma~\ref{transfer_parity}.
\end{proof}

By these lemmas, we can establish the following attainable and nicely attainable sequences, where $o$, $e$, and $e/0$ represent any positive odd integer, positive even integer, and non-negative even integer, respectively:

\begin{center} \label{nicely_attainable_table}
\begin{tabular}{p{3.7cm}|p{9.1cm}}
\emph{Nicely attain.\ seq.} & \emph{Corresponding well-behaved sequence of transfers}\\
\hline
$o$ & $v_{1}\rightarrow v_{2}$\\
\hline
$e, o, o, o, e$ & $v_{1}\rightarrow v_{4}\rightarrow v_{3}\rightarrow v_{2}\rightarrow v_{5}\rightarrow v_{6}$\\
\hline
$e, o, e, e, o, e$ & $v_{1}\rightarrow v_{4}\rightarrow v_{5}\rightarrow v_{2}\rightarrow v_{3}\rightarrow v_{6}\rightarrow v_{7}$\\
\hline
\parbox{3.7cm}{$e, e/0, \underbrace{o, \ldots , o}_{\clap{\parbox{2.7 cm}{\centering\footnotesize non-neg.\ even \#}}}, e/0, e$} & \parbox{9.1 cm}{$v_{1}\rightarrow [v_{2}\rightarrow v_{1}]\rightarrow [v_{4}\rightarrow v_{3}]\rightarrow\cdots$\\
$\rightarrow [v_{2k}\rightarrow v_{2k - 1}]\rightarrow v_{2k}\rightarrow v_{2k + 1}$}
\end{tabular}
\end{center}

\begin{center} \label{attainable_table}
\begin{tabular}{p{3.7cm}|p{9.1cm}} 
\emph{Attainable sequence} & \emph{Corresponding well-behaved sequence of transfers}\\
\hline
$e, \ldots , e$ & $v_{1}\rightarrow v_{2}\rightarrow\cdots\rightarrow v_{k - 1}\rightarrow v_{k}\rightarrow v_{k - 1}\rightarrow\cdots\rightarrow v_{1}$\\
\hline
$e, e/0, e, o$ & $v_{1}\rightarrow v_{2}\rightarrow v_{1}\rightarrow v_{4}\rightarrow v_{3}$\\
\hline
\parbox{3.7cm}{$e, e/0, \underbrace{o, \ldots , o}_{\clap{\parbox{2.7 cm}{\centering\footnotesize non-neg.\ even \#}}}$} & \parbox{9.1 cm}{$v_{1}\rightarrow [v_{2}\rightarrow v_{1}]\rightarrow [v_{4}\rightarrow v_{3}]\rightarrow\cdots$\\
$\rightarrow [v_{2k}\rightarrow v_{2k - 1}]$}
\end{tabular}
\end{center}

Each nicely attainable sequence above is also an attainable sequence, by removing the last transfer.  Also, note that the backwards double-8 transfer of Hrn\v{c}iar \& Haviar~\cite{diameter5} corresponds to the nicely attainable sequence $e, e, e, e$; we leave it out because $e, e, e, e$ is a special case of the last nicely attainable sequence above.

\subsection{Order of Leaves}

For the standard transfers $0\rightarrow n\rightarrow 1\rightarrow\cdots$, leaves are left behind in the same order as vertices appear in the sequence of transfers.  The following lemma makes this idea precise, and extends it to certain other transfers.  Note that we allow the list $v_{1}, \ldots , v_{m}$ to be long enough to include some of $c, \ldots , d$.

\begin{lemma} \label{order}
Let $(T, f, v_{1}, \ldots , v_{m}, a, b, c, d)$ be a transfer context, where the list $v_{1}, \ldots , v_{m}$ includes at least $l$ of the leaves $c, \ldots , d$, and let $v_{0}$ be a vertex with $f(v_{0}) = b$.  If a type-1 transfer $v_{1}\rightarrow v_{0}$ or $v_{1}\rightarrow v_{2}$ leaves behind $l$ of the leaves $c, \ldots , d$, then it leaves behind the $l$ leaves of $c, \ldots , d$ that occur first in the list $v_{1}, \ldots , v_{m}$.
\end{lemma}

\begin{proof}
Suppose $a < b$.  Since $v_{1}, \ldots , v_{m}$ includes some of $c, \ldots , d$, the vertices $v_{1}, \ldots , v_{m}$ must have labels $a, b - 1, a + 1, b - 2, \ldots$.  Suppose the leaves transferred by the transfer $v_{1}\rightarrow v_{0}$ or $v_{1}\rightarrow v_{2}$ have labels $c', \ldots , d'$.
\begin{itemize}
\item If we transfer $v_{1}\rightarrow v_{0}$, then $c' + d' = a + b$, so $c' - a = b - d'$.
\item If we transfer $v_{1}\rightarrow v_{2}$, then $c' + d' = a + b - 1$, so $c' - a + 1 = b - d'$.
\end{itemize}
Therefore, $c', d'$ occur consecutively in $v_{1}, \ldots , v_{m}$, so the leaves left behind are exactly the leaves that occur before $c', d'$ in the list $v_{1}, \ldots , v_{m}$.

If $a > b$, consider the labeling $g(v) = n - f(v)$, and apply the above.
\end{proof}

\noindent This lemma verifies the standard breadth-first construction (see Fig.~\ref{fig1}).

%%%%%%%%%%%%%%%%%%%%%%%%%%%%%%%%%%%%%%%%%%%%%%%%%%%%%%%

\section{Main Results}

We now prove our main result, Thm.~\ref{main_theorem} (see p.~\pageref{main_theorem}).  We can easily obtain any of these trees by transfers, up to the last level (see Fig.~\ref{fig1}).  The key idea is to rearrange subtrees to make the transfers at the last level possible.

\subsection{Proof of Thm.~\ref{main_theorem}(a): A Class of Diameter-6 Complete Trees}

\begin{proof}

Let $v$ have children $v_{1}, \ldots, v_{k}$, let each $v_{i}$ have children $v_{i, 1}, \ldots , v_{i, m_{k}}$, and let each $v_{i, j}$ have $n_{i, j}$ children.  Then we can represent $T$ as
$$((n_{1, 1}, \ldots , n_{1, m_{1}}), \ldots , (n_{k, 1}, \ldots , n_{k, m_{k}})).$$
By defining the $v_{i, j}$ differently, we can permute each $m_{i}$-tuple and also the list of $m_{i}$-tuples, giving a different sequence $n_{1, 1}, \ldots , n_{k, m_{k}}$; we say that the list of $m_{i}$-tuples \emph{corresponds} to any such sequence of integers.  It suffices to prove that each possible list of $m_{i}$-tuples corresponds to an attainable sequence.

We denote by $a/b$ any permutation of the $b$-tuple $(\overbrace{e, \ldots, e}^{\text{$a$ $e$'s}}, \overbrace{o, \ldots , o}^{\text{$b - a$ $o$'s}})$.  By the conditions, each $b$ is odd. 

\bigskip

\noindent\emph{Claim:} If each $m_{i}$-tuple is one of $0/1$, $1/3$, $2/3$, $3/5$, then the list of $m_{i}$-tuples corresponds to some sequence of integers of the form $n_{1}, \ldots , n_{m}, e, \ldots , e$, where $n_{1}, \ldots , n_{m}$ is nicely attainable.

\bigskip

\noindent\emph{Proof of Claim.}  We call a list of $m_{i}$-tuples \emph{nicely attainable} if it corresponds to some nicely attainable sequence of integers.  We list all minimal nicely attainable lists of $m_{i}$-tuples $0/1$, $1/3$, $2/3$, $3/5$ below, where brackets indicate nicely attainable subsequences (see Lemma~\ref{combine_attainable}, tables on p.~\pageref{nicely_attainable_table}):

\begin{center}
\begin{tabular}{|p{3.4 cm}|p{9.0 cm}|}
\hline
\multicolumn{2}{|c|}{Minimal nicely attainable lists of $m_{i}$-tuples}\\
\hline
$(0/1)$ & $((o))$\\
\hline
$(1/3, 1/3)$ & $((\underbrace{e, o, o), (o, e}, o))$\\
\hline
$(1/3, 3/5)$ & $((o, o, \underbrace{e), (e, e, e}, o, o))$\\
\hline
$(2/3, 2/3)$ & $((o, \underbrace{e, e), (e, e}, o))$\\
\hline
$(2/3, 3/5, 3/5)$ & $((o, \underbrace{e, e), (e, e}, o, o, \underbrace{e), (e, e, e}, o, o))$\\
\hline
$(3/5, 3/5, 3/5, 3/5)$ & $((o, o, \underbrace{e, e, e), (e}, o, o, \underbrace{e, e), (e, e}, o, o, \underbrace{e), (e, e, e}, o, o))$\\
\hline
\end{tabular}
\end{center}

We call a list of $m_{i}$-tuples $0/1$, $1/3$, $2/3$, $3/5$ \emph{irreducible} if no subset is nicely attainable.  We list all irreducible lists of $m_{i}$-tuples below; the right column shows that each corresponds to a sequence of the desired form:

\begin{center}
\begin{tabular}{|p{3.4 cm}|p{9.0 cm}|}
\hline
\multicolumn{2}{|c|}{Irreducible lists of $m_{i}$-tuples}\\
\hline
$\varnothing$ & $\varnothing$\\
\hline
$(1/3)$ & $((o, o, e))$\\
\hline
$(2/3)$ & $((o, e, e))$\\
\hline
$(3/5)$ & $((o, o, e, e, e))$\\
\hline
$(1/3, 2/3)$ & $((\underbrace{e, o, o), (o, e}, e))$\\
\hline
$(2/3, 3/5)$ & $((o, \underbrace{e, e), (e, e}, o, o, e))$\\
\hline
$(3/5, 3/5)$ & $((o, o, \underbrace{e, e, e), (e}, o, o, e, e))$\\
\hline
$(3/5, 3/5, 3/5)$ & $((o, o, \underbrace{e, e, e), (e}, o, o, \underbrace{e, e), (e, e}, o, o, e))$\\
\hline
\end{tabular}
\end{center}
\bigskip

By the definition of an irreducible list of $m_{i}$-tuples, we can permute any list of $m_{i}$-tuples $0/1$, $1/3$, $2/3$, $3/5$ into a list of the following form:
$$\underbrace{a_{1}/b_{1}, \ldots , a_{i_{1}}/b_{i_{1}}}_{\text{nicely attainable}}, \ldots , \underbrace{a_{i_{k - 1} + 1}/b_{i_{k - 1} + 1}, \ldots , a_{i_{k}}/b_{i_{k}}}_{\text{nicely attainable}}, \underbrace{a_{i_{k} + 1}/b_{i_{k} + 1}, \ldots , a_{i_{k + 1}}/b_{i_{k + 1}}}_{\text{irreducible}}$$
Therefore, any such list corresponds to a sequence of the desired form.

\bigskip

\noindent\emph{Claim:} If each $m_{i}$-tuple is one of $0/1$, $1/3$, $2/3$, $3/5$, $1/1$, then the list of $m_{i}$-tuples corresponds to some attainable sequence of integers.

\bigskip

\noindent\emph{Proof of Claim.} Permute the $m_{i}$-tuples so that all $1/1$'s are at the end.

\bigskip

We now prove the desired result.  We associate each general $a/b$ with one of $0/1$, $1/3$, $2/3$, $3/5$, $1/1$, giving five classes of $m_{i}$-tuples $a/b$:
\begin{center}
\begin{tabular}{l|c}
\multicolumn{1}{c|}{\emph{If...}} & \emph{Then associate $a/b$ with...}\\
\hline
$a < b$, $a\equiv 0\pmod{4}$, & $0/1$.\\
\hline
$a < b$, $a\equiv 1\pmod{4}$, & $1/3$.\\
\hline
$a < b$, $a\equiv 2\pmod{4}$, & $2/3$.\\
\hline
$a < b$, $a\equiv 3\pmod{4}$, & $3/5$.\\
\hline
$a = b$, & $1/1$.
\end{tabular}
\end{center}

Consider repeating the proofs of the claims above, but replacing each of $0/1$, $1/3$, $2/3$, $3/5$, $1/1$ with a general $a/b$ in its class.  To show that the proofs still hold, it suffices to prove the following three statements:

\begin{enumerate}[(1)]
\item The nicely attainable lists of $m_{i}$-tuples are still nicely attainable.
\item The irreducible lists of $m_{i}$-tuples still correspond to sequences $n_{1}, \ldots , n_{m}, e, \ldots , e$, where $n_{1}, \ldots , n_{m}$ is nicely attainable.
\item The lists associated with $1/1$ correspond to $e, \ldots , e$.
\end{enumerate}

(3) is clear. Consider (1) and (2), which only involve $m_{i}$-tuples associated with $0/1$, $1/3$, $2/3$, $3/5$.  We can obtain these general $m_{i}$-tuples from $0/1$, $1/3$, $2/3$, $3/5$ by repeatedly inserting $o$ or $e, e, e, e$.  Consider the following modified versions of the tables above:

\begin{center}
\begin{tabular}{|p{3.4 cm}|p{9.0 cm}|}
\hline
\multicolumn{2}{|c|}{Minimal nicely attainable lists of $m_{i}$-tuples}\\
\hline
$(0/1)$ & $((\ldots , o))$\\
\hline
$(1/3, 1/3)$ & $((\ldots, \underbrace{e, o, o), (o, e},\ldots , o))$\\
\hline
$(1/3, 3/5)$ & $((\ldots, o, o, \underbrace{e), (e, e, e}, \ldots, o, o))$\\
\hline
$(2/3, 2/3)$ & $((\ldots , o, \underbrace{e, e), (e, e},\ldots , o))$\\
\hline
$(2/3, 3/5, 3/5)$ & $((\ldots , o, \underbrace{e, e), (e, e},\ldots , o, o, \underbrace{e), (e, e, e}, \ldots, o, o))$\\
\hline
$(3/5, 3/5, 3/5, 3/5)$ &
\hangindent=3.8cm
$((\ldots , o, o, \underbrace{e, e, e), (e}, \ldots , o, o, \underbrace{e, e), (e, e},\hfill$
$\ldots , o, o, \underbrace{e), (e, e, e}, \ldots , o, o))$\\
\hline
\end{tabular}
\end{center}

\begin{center}
\begin{tabular}{|p{3.4 cm}|p{9.0 cm}|}
\hline
\multicolumn{2}{|c|}{Irreducible lists of $m_{i}$-tuples}\\
\hline
$\varnothing$ & $((\ldots))$\\
\hline
$(1/3)$ & $((\ldots , o, o, e))$\\
\hline
$(2/3)$ & $((\ldots , o, e, e))$\\
\hline
$(3/5)$ & $((\ldots , o, o, e, e, e))$\\
\hline
$(1/3, 2/3)$ & $((\ldots , \underbrace{e, o, o), (o, e},\ldots , e))$\\
\hline
$(2/3, 3/5)$ & $((\ldots , o, \underbrace{e, e), (e, e}, \ldots , o, o, e))$\\
\hline
$(3/5, 3/5)$ & $((\ldots , o, o, \underbrace{e, e, e), (e}, \ldots , o, o, e, e))$\\
\hline
$(3/5, 3/5, 3/5)$ & $((\ldots , o, o, \underbrace{e, e, e), (e}, \ldots , o, o, \underbrace{e, e), (e, e}, \ldots , o, o, e))$\\
\hline
\end{tabular}
\end{center}

\bigskip

By inserting each $o$ or $e, e, e, e$ at the dots, (1) and (2) also hold.
\end{proof}

\subsection{Proof of Thm.~\ref{main_theorem}(b): A Class of Diameter-6 Non-Complete Trees}

\begin{proof}
We adapt the proof above.  We again represent trees by lists of $m_{i}$-tuples, denoting by $a/b$ any permutation of the $b$-tuple of the form
\begin{align*}
(e/0, \overbrace{e, \ldots, e}^{\text{$a - 1$ $e$'s}}, \overbrace{o, \ldots , o}^{\text{$b - a$ $o$'s}})\qquad & \text{if $a > 1$}\\
(e, \overbrace{o, \ldots , o}^{\text{$b - 1$ $o$'s}})\qquad & \text{if $a = 1$}\\
(\overbrace{o, \ldots , o}^{\text{$b$ $o$'s}})\qquad & \text{if $a = 0$}
\end{align*}

Since our $m_{i}$-tuples now contain $e/0$'s, we no longer put all $1/1$'s at the end, and we classify our $m_{i}$-tuples differently.  We now associate each general $a/b$ with one of $0/1$, $1/3$, $2/3$, $3/3$, $1/1$:
\begin{center}
\begin{tabular}{l|c}
\multicolumn{1}{c|}{\emph{If...}} & \emph{Then associate $a/b$ with...}\\
\hline
$a\equiv 0\pmod{4}$, & $0/1$.\\
\hline
$a\equiv 1\pmod{4}$, $a < b$, & $1/3$.\\
\hline
$a\equiv 1\pmod{4}$, $a = b$, & $1/1$.\\
\hline
$a\equiv 2\pmod{4}$, & $2/3$.\\
\hline
$a\equiv 3\pmod{4}$, & $3/3$.
\end{tabular}
\end{center}

We list our minimal nicely attainable lists of $m_{i}$-tuples below:

\begin{center}
\begin{tabular}{|p{3.4 cm}|p{9.0 cm}|}
\hline
\multicolumn{2}{|c|}{Minimal nicely attainable lists of $m_{i}$-tuples}\\
\hline
$(0/1)$ & $((\ldots , o))$\\
\hline
$(1/3, 1/3)$ & $((\ldots, \underbrace{e, o, o), (o, e},\ldots , o))$\\
\hline
$(1/3, 3/3)$ & $((\ldots, o, o, \underbrace{e), (e/0, e, e}, \ldots))$\\
\hline
$(2/3, 2/3)$ & $((\ldots , o, \underbrace{e, e/0), (e/0, e},\ldots , o))$\\
\hline
$(3/3, 1/1)$ & $((\ldots, \underbrace{e, e/0, e), (e}, \ldots))$\\
\hline
$(1/3, 2/3, 1/1)$& $((\ldots, o, o, \underbrace{e), (\ldots , e), (e/0, e}, \ldots , o))$\\
\hline
$(2/3, 3/3, 3/3)$ & $((\ldots , o, \underbrace{e, e/0), (e/0, e}, \ldots , \underbrace{e), (e/0, e, e}, \ldots))$\\
\hline
$(2/3, 1/1, 1/1)$ & $((\ldots , o, \underbrace{e, e/0), (\ldots , e), (\ldots, e}))$\\
\hline
$(3/3, 3/3, 3/3, 3/3)$ &
\hangindent=4.9cm
$((\ldots, \underbrace{e, e/0, e), (e}, \ldots , \underbrace{e, e/0), (e/0, e},\hfill$
$\ldots , \underbrace{e), (e/0, e, e}, \ldots))$\\
\hline
$(1/3, 1/1, 1/1, 1/1)$ & $((\ldots , o, o, \underbrace{e), (\ldots , e), (\ldots , e), (\ldots , e}))$\\
\hline
$(1/1, 1/1, 1/1, 1/1)$ & $((\ldots , \underbrace{e), (\ldots , e), (\ldots , e), (\ldots, e}))$\\
\hline
\end{tabular}
\end{center}

As before, we call a list of $m_{i}$-tuples \emph{irreducible} if no subset is nicely attainable.  We list all irreducible lists of $m_{i}$-tuples below; the right column shows that each corresponds to a sequence of the desired form:

\begin{center}
\begin{tabular}{|p{3.4 cm}|p{9.0 cm}|}
\hline
\multicolumn{2}{|c|}{Irreducible lists of $m_{i}$-tuples}\\
\hline
$\varnothing$ & $((\ldots))$\\
\hline
$(1/3)$ & $((\ldots , o, o, e))$\\
\hline
$(2/3)$ & $((\ldots , o, e, e/0))$\\
\hline
$(3/3)$ & $((\ldots , e, e, e/0))$\\
\hline
$(1/1)$ & $((\ldots, e))$\\
\hline
$(1/3, 2/3)$ & $((\ldots , \underbrace{e, o, o), (o, e},\ldots , e/0))$\\
\hline
$(1/3, 1/1)$ & $((\ldots , o, o, e), (\ldots , e))$\\
\hline
$(2/3, 3/3)$ & $((\ldots , o, \underbrace{e, e/0), (e/0, e}, \ldots , e))$\\
\hline
$(2/3, 1/1)$ & $((e), (e/0, e, o))$\hfill(see tables on p.~\pageref{attainable_table})\\
\hline
$(3/3, 3/3)$ & $((\ldots , \underbrace{e, e/0, e), (e}, \ldots , e, e/0))$\\
\hline
$(1/1, 1/1)$ & $((\ldots, e), (\ldots , e))$\\
\hline
$(1/3, 1/1, 1/1)$ & $((\ldots , o, o, e), (\ldots , e), (\ldots , e))$\\
\hline
$(3/3, 3/3, 3/3)$ & $((\ldots , \underbrace{e, e/0, e), (e}, \ldots , \underbrace{e, e/0), (e/0, e}, \ldots , e))$\\
\hline
$(1/1, 1/1, 1/1)$ & $((\ldots , e), (\ldots, e), (\ldots , e))$\\
\hline
\end{tabular}
\end{center}
\bigskip

Therefore, any list of $m_{i}$-tuples $0/1$, $1/3$, $2/3$, $3/3$, $1/1$ corresponds to some attainable sequence of integers.  Consider replacing each of $0/1$, $1/3$, $2/3$, $3/3$, $1/1$ with a general $a/b$ in its class.  It suffices to prove the following:
\begin{enumerate}[(1)]
\item The nicely attainable lists of $m_{i}$-tuples are still nicely attainable.
\item The irreducible lists of $m_{i}$-tuples still correspond to attainable sequences.
\end{enumerate}

We can obtain these general $m_{i}$-tuples from $0/1$, $1/3$, $2/3$, $3/3$, $1/1$ by repeatedly inserting $o$ or a permutation of $e/0, e, e, e$.  For each list of $m_{i}$-tuples other than $(2/3, 1/1)$, it is clear that we can insert $o$ at the dots, but for inserting a permutation of $e/0, e, e, e$ at the dots we have three cases:

\bigskip

\noindent\emph{Case 1: Dots within a nicely attainable string $\underbrace{e, e/0, e/0, e}$.}

\noindent\emph{Proof of Case 1.}  Inserting permutations of $e/0, e, e, e$ produces a string of $e$'s and $e/0$'s with length a multiple of four.  We claim that we can choose the permutations so that the resulting string is of the form
$$\underbrace{e, e/0, e/0, e}, \ldots , \underbrace{e, e/0, e/0, e}$$
Since two of every four consecutive terms of this string are $e/0$'s, our claim is correct, and we can choose the permutations appropriately.

\bigskip

\noindent\emph{Case 2: Dots within a sequence of trailing $e$'s.}

\bigskip

\noindent\emph{Proof of Case 2.}  Inserting permutations of $e/0, e, e, e$ produces a longer string of trailing $e$'s.  We claim that we can choose the permutations so that the resulting string is of one of the following forms:
\begin{align*}
\underbrace{e, e/0, e/0, e}, \ldots , \underbrace{e, e/0, e/0, e}&\\
\underbrace{e, e/0, e/0, e}, \ldots , \underbrace{e, e/0, e/0, e}&, e/0\\
\underbrace{e, e/0, e/0, e}, \ldots , \underbrace{e, e/0, e/0, e}&, e, e/0\\
\underbrace{e, e/0, e/0, e}, \ldots , \underbrace{e, e/0, e/0, e}&, e, e, e/0
\end{align*}
Since at least one of every four consecutive terms of these strings are $e/0$'s, our claim is correct, and we can choose the permutations appropriately.

\bigskip

\noindent\emph{Case 3: All other dots.}

\bigskip

\noindent\emph{Proof of Case 3.}  We can insert $e, e/0, e, e$, which is nicely attainable.

\bigskip

Finally, consider lists of $m_{i}$-tuples associated with $(2/3, 1/1)$.  Each such list of $m_{i}$-tuples other than $(2/3, 1/1)$ itself can be obtained from one of the following by repeatedly inserting $o$ or a permutation of $e/0, e, e, e$:
\begin{center}
\begin{tabular}{|p{3.4 cm}|p{9.0 cm}|}
\hline
\multicolumn{2}{|c|}{Lists of $m_{i}$-tuples associated with $(2/3, 1/1)$}\\
\hline
$(2/3, 5/5)$ & $((\ldots , o, \underbrace{e, e/0), (e/0, e}, \ldots , e, e, e))$\\
\hline
$(2/5, 1/1)$ & $((\ldots , \underbrace{e), (o, o, o, e}, \ldots , e/0))$\\
\hline
$(6/7, 1/1)$ & $((\ldots , o, \underbrace{e, e/0, e, e}, e, e), (\ldots , e))$\\
\hline
\end{tabular}
\end{center}
Therefore, (1) and (2) hold, and the theorem is proved.
\end{proof}

\subsection{Proof of Thm.~\ref{main_theorem}(c): A Class of Diameter-\texorpdfstring{$2r$}{2r} Complete Trees}

\begin{proof}
We generalize our representation of depth-3 trees above to all trees $S$ as follows, where $S$ is a complete rooted tree with root $v$:
\begin{itemize}
\item If $S$ has depth 0 or 1, then we denote $S$ by $n$, where $v$ has $n$ children.
\item If $S$ has depth greater than 1, then we denote $S$ by $(S_{1}, \ldots , S_{k})$, where $S_{1}, \ldots , S_{k}$ are the connected components of $S\backslash v$.
\end{itemize}
By the conditions, if $S$ is a subtree of a tree satisfying the conditions of Thm.~\ref{main_theorem}(c), then our representation of $S$ has no 0's, and each set of parentheses encloses an odd number of trees.

Our intuition is that a tree $S$ functions in certain basic ways within larger trees, depending on the number of even integers in our representation of $S$:

\begin{center}
\begin{tabular}{l|p{9.8cm}}
\emph{\# even ints.} & \emph{$S$ can function as...}\\
\hline
$0\pmod{4}$ & a string of $o$'s.\\
\hline
$1\pmod{4}$ & a string of $o$'s, with one $e$ at the beginning or end.\\
\hline
$2\pmod{4}$ & a string of $o$'s, with two consecutive $e$'s in the middle.\\
\hline
$3\pmod{4}$ & \parbox{9.8cm}{a string of $o$'s, with two consecutive $e$'s in the middle,\\ and one $e$ at the beginning or end.}
\end{tabular}
\end{center}

To formalize this intuition, we define the \emph{ending} of a sequence, so that certain sequences have ending $\varnothing$, $E1$, $E2$, $E2'$, or $E3$, as follows:

\begin{center}
\begin{tabular}{l|p{4.7cm}|p{5.4cm}}
\emph{Ending} & \parbox{4.7cm}{\emph{Sequence of integers\\ (dfs.\ for Thm.\ 5(c))}} & \parbox{5.3cm}{\emph{Sequence of integers\\ (adjusted dfs.\ for Thm.\ 5(d))}}\\
\hline
$\varnothing$ & $\underbrace{n_{1}, \ldots , n_{m}}_{\clap{\parbox{2.5 cm}{\centering\footnotesize nicely att.}}}$ & (same)\\
\hline
$E1$ & $\underbrace{n_{1}, \ldots , n_{m}}_{\clap{\parbox{2.5 cm}{\centering\footnotesize nicely att.}}}, e$ & (same)\\
\hline
$E2$ & $\underbrace{n_{1}, \ldots , n_{m}}_{\clap{\parbox{2.5 cm}{\centering\footnotesize nicely att.}}}, e, e, \underbrace{o, \ldots , o}_{\clap{\parbox{2.5 cm}{\centering\footnotesize non-neg.\ even \#}}}$ & $\underbrace{n_{1}, \ldots , n_{m}}_{\clap{\parbox{2.5 cm}{\centering\footnotesize nicely att.}}}, e, e/0, \underbrace{o, \ldots , o}_{\clap{\parbox{2.5 cm}{\centering\footnotesize non-neg.\ even \#}}}$\\
\hline
$E2'$ & $\underbrace{n_{1}, \ldots , n_{m}}_{\clap{\parbox{2.5 cm}{\centering\footnotesize nicely att.}}}, e, e, \underbrace{o, \ldots , o}_{\clap{\parbox{2.5 cm}{\centering\footnotesize odd \#}}}$ & $\underbrace{n_{1}, \ldots , n_{m}}_{\clap{\parbox{2.5 cm}{\centering\footnotesize nicely att.}}}, e, e/0, \underbrace{o, \ldots , o}_{\clap{\parbox{2.5 cm}{\centering\footnotesize odd \#}}}$\\
\hline
$E3$ & $\underbrace{n_{1}, \ldots , n_{m}}_{\clap{\parbox{2.5 cm}{\centering\footnotesize nicely att.}}}, e, e, \underbrace{o, \ldots , o}_{\clap{\parbox{2.5 cm}{\centering\footnotesize non-neg.\ even \#}}}, e$ & $\underbrace{n_{1}, \ldots , n_{m}}_{\clap{\parbox{2.5 cm}{\centering\footnotesize nicely att.}}}, e, e/0, \underbrace{o, \ldots , o}_{\clap{\parbox{2.5 cm}{\centering\footnotesize non-neg.\ even \#}}}, e/0$\\
\end{tabular}
\end{center}

We associate a tree $S$ with a pair of endings $(\mathcal{E}_{1}, \mathcal{E}_{2})$ if $S$ corresponds to a sequence $\mathcal{S}$, such that appending $\mathcal{S}$ to any sequence with ending $\mathcal{E}_{1}$ gives a sequence with ending $\mathcal{E}_{2}$.  We can then combine trees; if $S_{1}, S_{2}$ are associated with $(\mathcal{E}_{1}, \mathcal{E}_{2}), (\mathcal{E}_{2}, \mathcal{E}_{3})$, respectively, then $(S_{1}, S_{2})$ is associated with $(\mathcal{E}_{1}, \mathcal{E}_{3})$.

Our intuition suggests the following claim:

\bigskip

\noindent\emph{Claim: If our representation of $S$ has no 0's, and each set of parentheses encloses an odd number of trees, then $S$ is associated with the following pairs, depending on the number of even integers in its representation:}
\begin{center}
\begin{tabular}{l|l}
\emph{\# even ints.} & \emph{Pairs of endings associated with $S$}\\
\hline
$0\pmod{4}$ & $(\varnothing, \varnothing), (E2, E2'), (E2', E2)$\\
\hline
$1\pmod{4}$ & $(\varnothing, E1), (E1, E2), (E2, E3), (E3, \varnothing)$\\
\hline
$2\pmod{4}$ & $(\varnothing, E2), (\varnothing, E2'), (E2, \varnothing), (E2', \varnothing)$\\
\hline
$3\pmod{4}$ & $(\varnothing, E3), (E1, \varnothing), (E2, E1), (E3, E2)$
\end{tabular}
\end{center}

\noindent\emph{Proof of Claim.}  We may assume that each set of parentheses in our representation of $S$ encloses one or three terms, by adding parentheses:
$$(S_{1}, \ldots , S_{k})\qquad\rightarrow\qquad(\cdots ((S_{1}, S_{2}, S_{3}), S_{4}, S_{5}), \ldots , S_{k})$$

It is easy to verify that $e$ and $o$ satisfy the claim, so by induction it suffices to prove by casework that if $S_{1}, S_{2}, S_{3}$ satisfy the claim, then $(S_{1}, S_{2}, S_{3})$ does also.  We omit the details, but the following observations are helpful:
\begin{itemize}
\item If $S_{1}, S_{2}$ each have an odd number of $e$'s, then $(S_{1}, S_{2})$ is associated with $(\varnothing, \varnothing)$ and $(E2, E2)$, or $(\varnothing, E2)$ and $(E2, \varnothing)$.
\item If $S_{1}, S_{2}$ each have an even number of $e$'s, then $(S_{1}, S_{2})$ is associated with $(\varnothing, \varnothing)$ and $(E2, E2)$, or $(\varnothing, E2)$ and $(E2, \varnothing)$.
\end{itemize}

Now we return to the original problem.  If the number of even integers in our representation of $S$ is not $3\pmod{4}$, then $S$ is associated with one of $(\varnothing, \varnothing), (\varnothing, E1), (\varnothing, E2)$.  Since the sequences associated with $\varnothing, E1, E2$ are attainable, $S$ is associated with an attainable sequence, as desired.
\end{proof}

\subsection{Proof of Thm.~\ref{main_theorem}(d): A Class of Diameter-\texorpdfstring{$2r$}{2r} Non-Complete Trees}

\begin{proof}
We use the same proof, adjusting the definitions of $E2, E2', E3$ as noted above to allow $e/0$'s, and including $(e, e, 0)$, $(e, o, 0)$ as base cases.
\end{proof}

\subsection{Proof of Thm.~\ref{main_theorem}(e): A Class of Diameter-6 Odd-Children Trees}

\begin{proof}
Rearrange the subtrees of $T$ at $v$ according to the numbers of internal vertices and leaves of distance 2 from $v$ in each subtree, as follows:
\begin{enumerate}[(1)]
\item positive odd number of internal vertices, no leaves
\item positive odd number of internal vertices, positive even number of leaves
\item positive even number of internal vertices, positive odd number of leaves
\item no internal vertices, positive odd number of leaves
\end{enumerate}
Let the neighbors of $v$ be $v_{1}, \ldots , v_{m}$, such that the subtrees are as follows:
$$\underbrace{v_{1}, \ldots , v_{i_{1} - 1}}_{\text{(1)}}, \underbrace{v_{i_{1}}, \ldots , v_{i_{2} - 1}}_{\text{(2)}},\underbrace{v_{i_{2}}, \ldots , v_{i_{3} - 1}}_{\text{(3)}},\underbrace{v_{i_{3}}, \ldots , v_{m}}_{\text{(4)}}$$
Starting with a gracefully labeled star with central vertex labeled 0, perform a $0\rightarrow n$ transfer, leaving $m$ vertices adjacent to the root.  Call them $v_{1}, \ldots , v_{m}$, in the order they would appear in a transfer context $(T, f, v_{1}, \ldots , v_{m}, n, 0, c, d)$.  Then perform the transfers
\begin{center}
\setlength{\tabcolsep}{1.35pt}
\begin{tabular}{lclclcl}
$v_{1}$ & $\rightarrow$ & $v_{2}$ & $\rightarrow\cdots\rightarrow$ & $v_{i_{3} - 2}$ & $\rightarrow$ & $v_{i_{3} - 1}$\\
$v_{i_{3} - 1}$ & $\rightarrow$ & $v_{i_{3} - 2}$ & $\rightarrow\cdots\rightarrow$ & $v_{i_{1} + 1}$ & $\rightarrow$ & $v_{i_{1}}$\\
$v_{i_{1}}$ & $\rightarrow$ & $v_{i_{1} + 1}$ & $\rightarrow\cdots\rightarrow$ & $v_{m - 1}$ & $\rightarrow$ & $v_{m}$
\end{tabular}
\end{center}

Thinking of this sequence as leaving behind an odd number of leaves once at each step, and twice at the turns at $v_{i_{3} - 1}, v_{i_{1}}$ (see Lemma~\ref{transfer_parity}), we can leave behind the following numbers of leaves at each step (referring to the internal vertices and leaves adjacent to the $v_{j}$ in each subtree):
\begin{center}
\setlength{\tabcolsep}{1.35pt}
\begin{tabular}{rclp{6cm}}
$v_{1},$ & $ \ldots ,$ & $v_{i_{1} - 1}$ & \qquad all of the internal vertices\\
$v_{i_{1}},$ & $ \ldots ,$ & $v_{i_{2} - 1}$ & \qquad all of the internal vertices\\
$v_{i_{2}},$ & $ \ldots ,$ & $v_{i_{3} - 1}$ & \qquad some of the internal vertices\\
$v_{i_{3} - 1},$ & $ \ldots ,$ & $v_{i_{2}}$ & \qquad rest of the internal vertices\\
$v_{i_{2} - 1},$ & $ \ldots ,$ & $v_{i_{1}}$ & \qquad some of the leaves\\
$v_{i_{1}},$ & $ \ldots ,$ & $v_{i_{2} - 1}$ & \qquad rest of the leaves\\
$v_{i_{2}},$ & $ \ldots ,$ & $v_{i_{3} - 1}$ & \qquad all of the leaves\\
$v_{i_{3}},$ & $ \ldots ,$ & $v_{m}$ & \qquad all of the leaves
\end{tabular}
\end{center}

After these transfers, the leaves left behind are in order (see Lemma~\ref{order}).  Therefore, we can obtain $T$ by performing the following sequence of transfers, stopping at the last leaf left at $v_{i_{2}}$ during the second pass:
$$v_{m}\rightarrow v_{m + 1}\rightarrow v_{m + 2}\rightarrow\cdots$$
Therefore, $T$ has a graceful labeling $f$ with $f(v) = 0$.
\end{proof}

We have now proved our main result, Thm.~\ref{main_theorem} (see p.~\pageref{main_theorem}).

%%%%%%%%%%%%%%%%%%%%%%%%%%%%%%%%%%%%%%%%%%%%%%%%%%%%%%%

\section{Conclusion}

In proving our result, we have both expanded the transfer technique, and consolidated it by removing the need for type-2 transfers. It is relevant to note that some previous papers have used type-2 transfers in novel ways not yet mentioned here, but it seems that even these uses of type-2 transfers could be replaced with type-1 papers.

For example, Hrn\v{c}iar \& Haviar~\cite{diameter5} use a type-2 $n\rightarrow 0$ transfer of branches as their last step. Specifically, see Fig.\ 11, 12 of Hrn\v{c}iar \& Haviar~\cite{diameter5}, where they use a $21\rightarrow 0$ transfer of the branches with vertices labeled $1, 3, 18, 20$. The same tree can be obtained by type-1 transfers as follows, beginning with a gracefully labeled star with central vertex labeled 0:
\begin{itemize}
\item Perform a type-1 transfer $0\rightarrow 21$ of the leaves labeled $2, \ldots , 19$.
\item Perform a type-1 transfer $21\rightarrow 0$ of the leaves labeled $3, \ldots , 18$.
\item Perform a type-1 transfer $0\rightarrow 21$ of the leaves labeled $4, \ldots , 17$.
\item Continue as in Hrn\v{c}iar \& Haviar~\cite{diameter5}.
\end{itemize}

An analogous remark applies to Mishra \& Panigrahi~\cite{mishra2005graceful, mishra2006graceful, mishra2008some, mishra2010some, mishra2011some}. The point here is not to criticize these arguments, but that one may choose to ignore type-2 transfers without losing any of the power of the transfer technique.

To extend our results to other diameter-6 trees, it is necessary to consider different sequences of transfers in higher levels, which have a complicated effect on the order of vertices in lower levels. Understanding this effect will be beneficial. Cahit's technique of \emph{canonic spiral labeling} is also relevant, since it produces some diameter-6 trees not included in our results (see slides 7-10 of Cahit~\cite{cahit_presentation}).  

%%%%%%%%%%%%%%%%%%%%%%%%%%%%%%%%%%%%%%%%%%%%%%%%%%%%%%%

\section{Acknowledgements}

\noindent Thanks to my advisor, Shiva Kintali, for his helpful guidance throughout.

%%%%%%%%%%%%%%%%%%%%%%%%%%%%%%%%%%%%%%%%%%%%%%%%%%%%%%%

\end{document}